\documentclass[12pt]{amsart}
\usepackage{amssymb}
\usepackage{hyperref}

\numberwithin{equation}{section}
\newtheorem{theorem}{Theorem}[section]

\newtheorem{lemma}[theorem]{Lemma}
\theoremstyle{proposition}
\newtheorem{proposition}[theorem]{Proposition}
\theoremstyle{corollary}

\theoremstyle{definition}

\newtheorem{remark}[theorem]{Remark}

 \DeclareMathOperator\rank{{\rm rk}}

\def\Sp{{\rm Sp}}

\newcommand\LG{{\rm G}}

\newcommand\R{\mathbb{R}}
\newcommand\SO{{\rm SO}}
\newcommand\SUxU[2]{ {\rm S(U(}#1)\,{\rm U(}#2{\rm ))} } 
\newcommand\SU{{\rm SU}}

\newcommand\Spin{{\rm Spin}}

\newcommand\U{{\rm U}}

\newcommand\g{{\mathfrak g}}

\newcommand\h{{\mathfrak h}}

\newcommand\product{\,{\cdot}\,}
\newcommand\p{{\mathfrak p}}

\newcommand\so{\mathfrak s \mathfrak o}

\title{Polar actions on symmetric spaces of higher rank}

\author[A. Kollross]{Andreas Kollross }

\author[A. Lytchak]{Alexander Lytchak}

\thanks{The second-named author was supported by a Heisenberg grant of the DFG
and by the SFB~``Groups, Geometry and Actions''}

\address{Universit\"{a}t Stuttgart, Fachbereich Mathematik, Institut f\"{u}r Geometrie und Topologie, Pfaffenwaldring 57,
D--70569 Stuttgart, Germany}

\address{Universit\"{a}t zu K\"{o}ln, Mathematisches Institut, Weyertal 86--90, D--50931 K\"{o}ln, Germany}

\email{kollross@mathematik.uni-stuttgart.de}

\email{alytchak@math.uni-koeln.de}

\date{\today}



\begin{document}
\maketitle

\begin{abstract}
We show that polar actions of cohomogeneity two on simple compact Lie groups of
higher rank, endowed with a biinvariant Riemannian metric, are hyperpolar.
Combining this with a recent result of the second-named author, we are able to
prove that polar actions induced by reductive algebraic subgroups in the
isometry group of an irreducible Riemannian symmetric space of higher rank are
hyperpolar. In particular, this result affirmatively settles the conjecture
that polar actions on irreducible compact symmetric spaces of higher rank are
hyperpolar.
\end{abstract}


\section{Introduction and main results}
\label{SecIntro}


A proper isometric action of a Lie group on a Riemannian manifold is called
\emph{polar} if there exists an immersed submanifold which meets every orbit
and each such intersection is orthogonal. Such a submanifold is called a
\emph{section} of the Lie group action. If there is a section which is flat in
its induced Riemannian metric, then the action is called \emph{hyperpolar}.

A typical example of a polar action is the following. Consider the orthogonal
group~${\rm O}(n)$, which acts on the space of real symmetric $n \times
n$-matrices by conjugation. If this space is endowed with the scalar product
given by $\langle X, Y \rangle ={\rm tr}(XY)$, then the action is isometric. It is a
well-known fact of linear algebra that every real symmetric matrix is conjugate
to a diagonal matrix, which means that the linear subspace consisting of the
diagonal matrices meets every orbit of the group action and it is easy to see
that this subspace meets the orbits orthogonally. Hence the action is polar (in
fact hyperpolar) with the subspace of diagonal matrices as a section. As this
example suggests, one may think of the elements in a section as canonical
representatives of the group orbits. The word ``polar'' is used because of the
analogy to the usual polar coordinates on~$\R^2$ which can be regarded as being
given by the polar action of $\rm O(2)$ on~$\R^2$, where the orbits are
concentric circles around the origin and any straight line through the origin
is a section. Indeed, polar actions on Riemannian manifolds are generalizations
of the usual polar, cylindrical, or spherical coordinates in Euclidean space to
general Riemannian manifolds. For surveys on polar actions, see~\cite{th1} and
\cite{th2}.

Polar actions are closely related to Riemannian symmetric spaces. For example,
if $G$ is the isometry group of a Riemannian symmetric space~$M$ and \[K = \{ g
\in G \mid g \cdot p = p \}\] is the isotropy subgroup at a point~$p \in M$ of
the $G$-action on~$M=G/K$, then the \emph{isotropy action} of~$M$, i.e.\ the
action of~$K$ on~$M$, is polar, in fact hyperpolar.

Indeed, this action is well known from the theory of symmetric
spaces~\cite{helgason}; the \emph{flats} of a Riemannian symmetric space are
the maximal totally geodesic and flat subspaces; all flats are congruent under
the action of the isometry group; the \emph{rank} of the symmetric space is
defined to be the dimension of the flats. The sections of the isotropy action
are given by those flats of~$M$ which contain the point~$p$. We say that a
symmetric space is of \emph{higher rank} if the rank of~$M$ is greater than
one.

The isotropy subgroup~$K$ also acts on the tangent space $T_pM$ by the
differentials of isometries. This linear representation of~$K$ is called the
\emph{isotropy representation} of the symmetric space~$M$; it is hyperpolar and
the sections are the linear subspaces of~$T_pM$ tangent to flats. Conversely,
the main result of~\cite{dadok} says that all orthogonal representations of
compact Lie groups on Euclidean space which are polar are given by this
construction. More precisely, Dadok shows in his article~\cite{dadok} that any
polar representation is \emph{orbit equivalent} to an isotropy representation
of a Riemannian symmetric space, i.e.\ after a suitable isometric
identification of the two representation spaces, the connected components of
the orbits coincide.

A special case of a Riemannian symmetric space is a connected compact Lie
group~$L$ equipped with a biinvariant Riemannian metric. These spaces are
called symmetric spaces of Type~II in~\cite{helgason}. In this case, left and
right multiplication with arbitrary group elements are isometries and hence the
group $L \times L$ acts isometrically on~$L$ by
\begin{equation}\label{EqLRAction}
(a,b) \cdot \ell := a \, \ell \, b^{-1}.
\end{equation}
In fact, the connected component of the isometry group of~$L$ with this metric is covered by $L \times L
=: G$. The isotropy subgroup $G_e$ at the identity element~$e \in L$ is given
by the diagonal $\{(\ell,\ell) \mid \ell \in L \} =: K$. The isotropy action of
the symmetric space~$L$ at the identity element~$e$ is the restriction of
(\ref{EqLRAction}) to~$K$ and hence coincides with the action of~$L$ on itself
by conjugation. The sections of this hyperpolar action are the maximal tori
of~$L$. It is a well-known fact of Lie theory that if $T \subseteq L$ is a
maximal torus, then any element of a connected compact Lie group is conjugate
to some element in~$T$, i.e.\ $T$~meets all the orbits of the action of~$L$ on
itself by conjugation (and it can be easily checked that $T$~meets the orbits
orthogonally). Indeed, in a symmetric space of Type~II, the flats containing
the identity element~$e$ are exactly the maximal tori. In particular, the rank
of a symmetric space of Type~II equals the dimension of the maximal tori.

Isotropy actions of Riemannian symmetric spaces provide examples of hyperpolar
actions, but there is also a more general construction which involves two (in
general) different symmetric spaces. Let $G$ be a compact semisimple Lie group
and let $\sigma$ and $\tau$ be two involutive automorphisms. Let $H$ and $K$ be
subgroups of~$G$ such that $G_o^\sigma \subseteq H \subseteq G^\sigma$ and
$G_o^\tau \subseteq K \subseteq G^\tau$, i.e.\ the groups $H$ and $K$ are open
subsets of the fixed point sets of $\sigma$ and $\tau$. Assume that $G/K$ is
endowed with the $G$-invariant metric induced by the negative of the Killing
form of~$\g$. In this way, $G/K$ becomes a Riemannian symmetric space whose
connected component of the isometry group is covered by $G$, see~\cite{helgason}. In particular, the
closed subgroup $H \subseteq G$ naturally acts by isometries on $G/K$ and this
action is in fact hyperpolar, see~\cite{hptt}. Actions which arise in this way
are called \emph{Hermann actions} since they were introduced in~\cite{hermann}.
Of course, the isotropy action described above is just the special case of this
construction where $H=K$

We show in this article that any polar action of cohomogeneity~two on a simple
compact classical Lie group~$L$ of higher rank and endowed with a biinvariant
Riemannian metric is hyperpolar. As we will point out below, this result
provides the last step in the proof of the following theorem, which had been
conjectured to be true by Biliotti in~\cite{biliotti}.

\begin{theorem}\label{ThMain}
Any polar action on an irreducible compact Riemannian symmetric space of higher
rank which has an orbit of positive dimension is hyperpolar.
\end{theorem}

Note that we have to exclude the case of actions all of whose orbits are
zero-dimensional from the hypothesis of the theorem, since such actions are
always polar (the whole space is the section) and they would otherwise provide
technical counterexamples.

Theorem~\ref{ThMain} has already been proved in a number of special cases. The
first result in this connection was obtained by Br\"{u}ck~\cite{brueck}, who showed
that polar actions with a fixed point on irreducible Riemannian symmetric
spaces (of compact or non-compact type) are hyperpolar. Podest\`{a} and
Thorbergsson~\cite{pth2} have shown that polar actions on the complex quadric
$\SO(n{+}2)\, /\, \SO(n) \, \SO(2)$ are hyperpolar. Their method relies on the
fact that polar actions on Hermitian symmetric spaces are coisotropic, i.e.\
the normal spaces to principal orbits are totally real. (In fact, they have
classified the coisotropic actions on the complex quadric.) Using the same
method, this result has been generalized by Biliotti and Gori~\cite{bg} to the
complex Grassmannians and then by Biliotti~\cite{biliotti} to all compact
irreducible Hermitian symmetric spaces, leading Biliotti~\cite{biliotti} to
formulate the statement of Theorem~\ref{ThMain} as a conjecture. However, this
technique is not applicable to other (i.e.\ non-Hermitian) symmetric spaces and
different methods are needed to prove the conjecture in a more general setting.

In the special cases of symmetric spaces of Type~I, i.e.\ Riemannian symmetric
spaces whose isometry group is a simple compact Lie group, the first-named
author has obtained a complete classification of polar actions~\cite{polar},
proving Theorem~\ref{ThMain} in this special case. The other class of
irreducible compact Riemannian symmetric spaces are the symmetric spaces of
Type~II, i.e.\ compact Lie groups~$L$ with biinvariant metric. These spaces
have a non-simple isometry group; in fact, the connected component of the 
isometry group is covered by~$L \times L$. 
The fact that the isometry group is non-simple makes it considerably
more difficult to apply the methods of~\cite{polar} in the case of symmetric
spaces of Type~II. However, the first-named author could still prove
Theorem~\ref{ThMain} in the special case of exceptional compact Lie groups with
a biinvariant metric~\cite{exc}, again confirming the conjecture of Biliotti.

Recently, the second-named author obtained a result which generalizes
Theorem~\ref{ThMain} to the more general case of \emph{polar foliations}, where
the leaves of the foliation do not necessarily have to be homogeneous. He has
shown that polar (singular Riemannian) foliations of irreducible compact
symmetric spaces of higher rank are hyperpolar under the additional assumption
that the codimension of the foliation is at least three~\cite{lytchak}.
(See~\cite{th2} for a survey article on singular Riemannian foliations.) In
order to finally complete the proof of Theorem~\ref{ThMain}, it now suffices to
show that a polar action of cohomogeneity~two on a classical compact Lie
group~$L$ endowed with a biinvariant Riemannian metric is hyperpolar if
$\rank(L) \ge 2$. This is done in the present article by classifying all polar
cohomogeneity two actions on classical compact Lie groups with biinvariant
metric; all such actions turn out to be hyperpolar. See the last paragraph of
Section~\ref{SecCriterion} below for a summary of our proof.

Note that Theorem~\ref{ThMain} does not generalize directly to non-compact
symmetric spaces. In fact, there are counterexamples of polar actions with
non-flat sections on non-compact symmetric spaces of higher rank, see
\cite[Proposition~4.2]{bdrt}. However, an analogous statement as in
Theorem~\ref{ThMain} still holds for actions on non-compact irreducible spaces
if one requires the action to be given by a reductive algebraic subgroup of the
isometry group. (For example, semisimple subgroups and compact subgroups of a
semisimple Lie group are reductive algebraic subgroups, see~\cite{ov}.) Hence
we also obtain the following.

\begin{theorem}\label{ThMain2}
Let $M$ be an irreducible symmetric space. Let $H$ be a reductive algebraic
subgroup of the isometry group of~$M$. If the action of~$H$ on~$M$ is polar and
has orbits of positive dimension, then the action is hyperpolar or $M$ is of
rank one.
\end{theorem}

Theorem~\ref{ThMain2} follows immediately from Theorem~\ref{ThMain} and
\cite[Theorem~5.1]{dual}. It was shown in~\cite{dual} that hyperpolar actions
of reductive algebraic subgroups in the isometry group of irreducible symmetric
spaces are of cohomogeneity one or Hermann actions.


\section{Symmetric spaces of Type II and their isometry groups}
\label{SecTypeII}


Following \cite[Ch.\,X]{helgason}, there are two types of irreducible compact
Riemannian symmetric spaces. The spaces of \emph{Type}~I are those with simple
compact isometry group; the spaces of \emph{Type}~II are given by simple
compact Lie groups equipped with a biinvariant Riemannian metric. In the
following, we study isometric group actions of compact Lie groups on Riemannian
symmetric spaces of Type~II.

Let $L$ be a compact connected simple Lie group equipped with a biinvariant
Riemannian metric. Any such biinvariant metric is unique up to a constant
scaling factor. To study polar actions, the choice of this scaling factors is
of no relevance and we may thus assume $L$ is equipped with the biinvariant
metric induced by the negative of the Killing form. Then $L$ together with this
metric is a Riemannian symmetric space~\cite{helgason}; indeed, $L$ is a
Riemannian homogeneous space (it acts isometrically on itself by left and right
translations) and the inversion map $g \mapsto g^{-1}$ turns out to be an
isometric geodesic symmetry at the identity element $e$~of $L$.

The connected component of the isometry group of this symmetric space is
covered by $L \times L$, see~\cite{helgason}. To study (effective) polar
actions on $L$, it therefore suffices to consider closed connected subgroups of
$L \times L$ which act non-transitively. The following proposition shows that we
may distinguish between two types of maximal non-transitive subgroups of~$L
\times L$.

\begin{proposition}\label{PropMaxSubGr}
Let $H \subset L \times L$ be a connected closed subgroup acting
non-transitively on $L$. Then $H$ is contained in one of the following subgroups
of $L \times L$.
\begin{enumerate}

\item A diagonal subgroup of the form
\begin{equation}\label{EqDiagSubgroup}
\Delta^\sigma  L := \{ (\ell, \sigma(\ell)) \mid \ell \in L \}
\end{equation}
where $\sigma$ is an automorphism of~$L$.

\item A subgroup of the form
\begin{equation}\label{EqTimes}
H_1 \times H_2 := \{(h_1,h_2) \mid h_1 \in H_1,\, h_2 \in H_2 \},
\end{equation}
where $H_1$ and $H_2$ are closed connected proper subgroups of~$L$.

\end{enumerate}
\end{proposition}

\begin{proof}
Follows from~\cite[Theorem~15.1]{dynkin1}, see also
\cite[Section~2.1]{hyperpolar}.
\end{proof}


\section{Criterion for polarity}
\label{SecCriterion}


The following criterion for polarity of isometric actions on symmetric spaces
of the compact type is well known, for a proof see for
instance~\cite[Proposition~4.1]{polar}. By $H \cdot e K$ we denote the
$H$-orbit through the point $e K \in G/K$ and by $N_{e K}(H \cdot e K)$ its
normal space at the point $eK$.

\begin{proposition}\label{PropPolCrit}
Let $G$ be a connected semisimple compact Lie group, let $\sigma$ be an
involutive automorphism of~$G$ and let $K \subset G$ be a closed subgroup such
that $G^\sigma_o \subseteq K \subseteq G^\sigma$; in particular, $G/K$ is a
Riemannian symmetric space of the compact type. Let $\g = {\mathfrak k} \oplus
\p$ be the decomposition into eigenspaces of $d\sigma_e$ and identify
$T_{eK}G/K$ with $\p$ in the usual way. Let $H \subset G$ be a closed subgroup.
Assume the element $e K$ of~$G/K$ lies in a principal orbit of the $H$-action
on $G/K$. Then the following are equivalent.
\begin{enumerate}

\item The $H$-action on $G / K$ is polar with respect to any $G$-invariant
    Riemannian metric on~$G/K$.

\item The subspace $\nu := N_{e K}(H \cdot e K) \subseteq \p$ is a Lie
    triple system such that the Lie algebra $\nu \oplus [\nu, \nu]$
    generated by~$\nu$ is orthogonal to~$\h$ with respect to the negative 
    of the Killing form on~$\g$.

\end{enumerate}
In particular, the $H$-action on~$G/K$ is hyperpolar if and only if the
subspace $\nu$ is an abelian subalgebra of~$\p$.
\end{proposition}

The proposition implies that the polarity of an action is determined on the Lie
algebra level. The Lie triple system~$\nu$ which appears in the proposition
corresponds to the tangent space of a section in case of a polar action. (Note
that sections of polar actions are always totally geodesic
submanifolds~\cite{pt}.)

In this article, we will apply Proposition~\ref{PropPolCrit} exclusively to the
special case of a Riemannian symmetric space of Type~II. To this end, we use
the homogeneous presentation~$G/K$ of~$L$ where we define $G := L \times L$ and
where the isotropy subgroup~$K = G_e$ of the identity element~$e \in L$ for the
action (\ref{EqLRAction}) is the diagonal subgroup $\{ (\ell, \ell) \mid \ell
\in L \}$ of $L \times L$. In this case the spaces~$\mathfrak k$ and $\p$ as
defined in Proposition~\ref{PropPolCrit} are of the following form
\begin{equation}\label{EqAntiDiag}
\mathfrak k = \{ (X,X) \mid X \in \mathfrak{l} \}
\end{equation}
and
\begin{equation}\label{EqAntiDiag}
\p = \{ (X,-X) \mid X \in \mathfrak{l} \},
\end{equation}
where $\mathfrak{l}$ denotes the Lie algebra of~$L$.

Henceforth we will assume that $L$ is a simple compact Lie group of rank greater
than one and $H \subset L \times L$ is a closed connected subgroup which acts
polarly with non-flat sections and with cohomogeneity two on~$L$ by the
restriction of~(\ref{EqLRAction}) to~$H$. We will prove that no such subgroup
exists.

Our proof proceeds as follows. By Proposition~\ref{PropMaxSubGr} we may assume
that the group $H$ is contained in a group of the form~(\ref{EqDiagSubgroup})
or of the form~(\ref{EqTimes}). The first case has already been treated
in~\cite{exc} and hence may assume that $H$ is contained in a group of the
form~(\ref{EqTimes}). In Lemma~\ref{LmPolDiag} we show that if $H$ is of the
special form $H=\pi_1(H) \times \pi_2(H)$, where $\pi_1$ and $\pi_2$ denote the
canonical projections on the first or second factor of~$L \times L$, then the
action is hyperpolar. Thus we may assume in what follows that $H$ is a proper
subgroup of $\pi_1(H) \times \pi_2(H)$. Moreover, since the action of $\pi_1(H)
\times \pi_2(H)$ is not orbit equivalent to the $H$-action, it is either of
cohomogeneity one or transitive. However, in Section~\ref{SecSubCohOne} we show
that subactions of cohomogeneity one actions cannot be polar, ruling out the
first alternative. It remains to consider the case where $\pi_1(H) \times
\pi_2(H)$ acts transitively on $L$ by~(\ref{EqLRAction}) and where both
$\pi_1(H)$ and $\pi_2(H)$ are proper subgroups of~$L$. Using the classification
of transitive actions~\cite{oniscik}, it is shown in Section~\ref{SecSubTrans}
that there is no such action, completing the proof of Theorem~\ref{ThMain}.


\section{Subactions of $\sigma$-actions}
\label{SecSubSigma}


We start our classification with actions by (closed subgroups of) diagonal
subgroups in~$L$.

\begin{lemma}\label{LmNoDiag}
Let $L$ be a simple compact Lie group and let $\sigma$ be an automorphism
of~$L$. If a closed connected subgroup~$H$ of $\Delta^\sigma L$ acts polarly on
the simple compact connected Lie group~$L$ of rank greater than one, then $H =
\Delta^\sigma L$ or $H = \{e\}$; in particular, the action of~$H$ on~$L$ is
hyperpolar or trivial.
\end{lemma}

\begin{proof}
\cite[Proposition~12]{exc}.
\end{proof}

Note that the action of $\Delta^\sigma L$ on~$L$ is well known to be
hyperpolar. These actions were named \emph{$\sigma$-actions} in~\cite{hptt}.
From now on, it suffices to consider those subgroups of $L \times L$ which are
contained in groups of the form (\ref{EqTimes}). We will first look at a
special case.


\section{Groups without diagonal factors}
\label{SecNoDiagFactors}


In this section we consider the special case where the subgroup $H \subset L
\times L$ is such that $H = \pi_1(H) \times \pi_2(H)$. It is more or less an
immediate consequence of Proposition~\ref{PropPolCrit} that such an action is
hyperpolar.

\begin{lemma}\label{LmPolDiag}
Let $L$ be a simple compact Lie group and let $H_1, H_2 \subset L$ be closed
subgroups. Assume that the subgroup
\begin{equation}\label{EqSplitGrp}
H_1 \times H_2 = \{\, (h_1,h_2) \mid h_1 \in H_1,\, h_2 \in H_2\, \} \subseteq L \times L
\end{equation}
acts polarly and with cohomogeneity two on~$L$ by the restriction of the
action~(\ref{EqLRAction}). Then the $H_1 \times H_2$-action on $L$ is
hyperpolar.
\end{lemma}
\begin{proof}
We may assume that the identity element $e \in L$ lies in a principal orbit by
replacing $H$ with a suitable conjugate group, if necessary. Let $\nu$ be the
normal space at the point~$e$ to the $H$-orbit through~$e$ as defined in
Proposition~\ref{PropPolCrit}~(ii). By the hypothesis, the $H$-action is of
cohomogeneity two, hence there are two linearly independent vectors $v,w \in
\nu$. Furthermore, since the $H$-action is polar, it follows from
Proposition~\ref{PropPolCrit} that $[v,w]$ is orthogonal to~$\h$. Since the
elements of~$\p$ (and hence of~$\nu$) are of the special
form~(\ref{EqAntiDiag}), we have $v = (X,-X)$, $w = (Y,-Y)$ for some vectors
$X,Y \in \mathfrak{l}$. Consequently, we have $[v,w] = [(X,-X),(Y,-Y)] =
([X,Y],[X,Y])$. By our hypothesis, the group $H$ is of the special
form~(\ref{EqSplitGrp}) and hence an element $(A,B) \in \mathfrak{l} \times
\mathfrak{l}$ is orthogonal to~$\h$ if and only if $A$ is orthogonal to~$\h_1$
and $B$ is orthogonal to~$\h_2$. In particular, since the vector $[v,w] =
([X,Y],[X,Y])$ is orthogonal to~$\h$, we also have that the vector
$([X,Y],-[X,Y]) \in \p$ is orthogonal to~$\h$ and hence it is contained in the
normal space~$\nu$.

Assume now the three vectors $X$, $Y$ and $[X,Y]$ in~$\mathfrak{l}$ are
linearly independent. Then it follows that the three vectors $(X,-X)$,
$(Y,-Y)$, $([X,Y],-[X,Y])$ in~$\nu$ are linearly independent, too. Since $\nu$ is
the normal space at a principal orbit, it follows that the cohomogeneity of the
$H$-action is at least three, contradicting our hypothesis.

Thus it follows that the bracket~$[X,Y]$ is contained in the span of~$X$
and~$Y$. Hence the vectors $X$ and $Y$ span a two-dimensional subalgebra
of~$\mathfrak{l}$. However, the Lie group $L$ is compact and hence any
subalgebra of $\mathfrak{l}$ is reductive. In particular, any subalgebra
of~$\mathfrak{l}$ decomposes as a direct sum of an abelian and a semisimple Lie
algebra. By the classification of semisimple real Lie algebras, it is known
that any simple Lie algebra (and hence any non-trivial semisimple Lie algebra)
is at least of dimension three. It follows that the Lie algebra spanned by $X$
and $Y$ is abelian and Proposition~\ref{PropPolCrit} implies that the action is
hyperpolar.
\end{proof}

As a consequence of Lemmas~\ref{LmNoDiag} and \ref{LmPolDiag} we obtain the
following.

\begin{lemma}
Let $L$ be a simple compact Lie group and assume $H \subset L \times L$ is a
closed connected subgroup whose action on~$L$ is of cohomogeneity two and polar
with a non-flat section. Then there are closed connected proper subgroups $H_1$
and $H_2$ of~$L$ such that $H$ is contained in $H_1 \times H_2$ and such that
the action of $H_1 \times H_2$ on~$L$ given by $(h_1,h_2) \cdot \ell := h_1 \,
\ell \, h_2^{-1}$ is of cohomogeneity one or transitive.
\end{lemma}

\begin{proof}
Define $H_1$ and $H_2$ to be the images of~$H$ under the natural projections
onto the first and second factors of $L \times L$. Then $H_1$ and $H_2$ are
compact connected subgroups of~$L$ and $H$ is contained in $H_1 \times H_2$. If
$H_1 = L$ or $H_2 = L$, then $\h$ contains an ideal isomorphic to
$\mathfrak{l}$ and it follows from Lemma~\ref{PropMaxSubGr} that $H$ is
contained in a subgroup of the form~(\ref{EqDiagSubgroup}). Hence it follows
from Lemma~\ref{LmNoDiag} that the $H$-action on~$L$ is hyperpolar, a
contradiction. Thus $H_1$ and $H_2$ are proper subgroups of $L$. Since we have
$H \subseteq H_1 \times H_2$, it follows that the $H_1 \times H_2$-action
on~$L$ is cohomogeneity less or equal to two. If the cohomogeneity is two, then
the $H$-action and the $H_1 \times H_2$-action on~$L$ are orbit equivalent,
leading to a contradiction, since then by Lemma~\ref{LmPolDiag} the $H_1 \times
H_2$-action is hyperpolar. Thus the $H_1 \times H_2$-action on $L$ is of
cohomogeneity one or transitive.
\end{proof}

From now on, we may assume that the $H$-action on $L$ is a subaction of a
transitive or cohomogeneity one action given by a subgroup of the
form~(\ref{EqTimes}). (By the term \emph{subaction} we refer to an action which
is given by restricting a Lie group action to a closed subgroup.) We will prove
in the next two sections that no such subaction exists.


\section{Subactions of cohomogeneity one actions}
\label{SecSubCohOne}


\begin{proposition}\label{PropCohOne}
Let $M$ be a compact irreducible symmetric space of higher rank with a polar
and not hyperpolar action of a group~$H$  of cohomogeneity two. Then $H$ is not
a subgroup of a larger group~$H'$ which acts on~$M$ isometrically and with
cohomogeneity one.
\end{proposition}

\begin{proof}
Assume that  such a group $H'$ does  exist. We may assume that $H$ and $H'$ are
connected.  Denote by $\Delta$ the quotient space $M/H$.  By polarity of the
action of~$H$, $\Delta$ is a finite quotient of a section~$\Sigma$, which is a
totally geodesic submanifold of~$M$ and hence a two-dimensional symmetric
space. By assumption $\Sigma $ is non-flat. Thus it is  either the round sphere
or the round projective space, which, after rescaling, may be assumed to be of
constant curvature one. In any case, $\Delta $ is the quotient of the universal
covering $S^2$ of~$\Sigma$ by a finite Coxeter group~$W$.

In \cite[Section~7]{lytchak} it is shown that if the Coxeter group~$W$ is
reducible, the space $M$ must be of rank one. Thus we may assume that $W$ is
irreducible. Then $\Delta $ is a spherical triangle with angles $\pi /2, \pi/3,
\pi/m$, with $m=3$ or $m=4$. (In fact, the case $m=3$ can be excluded using the
ideas of \cite{lytchak}, cf.\ \cite{fgt}, but this will not be used in the
sequel).  Hence, the triangle $\Delta$  has two vertices $x$ and $y$ such that the
angles at these points are equal to $\pi /m$ with (possibly different) $m>2$.

Let $p \in M$ be any point in the $H$-orbit corresponding to~$x$.  Then the
isotropy group $H_p$  acts on the normal space $V$ to the orbit $H\cdot p$ with
cohomogeneity two, such that the quotient $V/H_p$ is the tangent space
to~$\Delta$ at~$x$, i.e.\ the cone over the interval of length~$\pi /m$. Hence
the action of~$H_p$ on $V$ is irreducible. Therefore, the larger isotropy group
$H' _p$ of~$H'$ at $p$ cannot act on a proper non-trivial subspace of~$V$. It
follows that the normal space to the orbit  $H' \cdot p$ at $p$ coincides with
$V$. Thus the orbits  $H' \cdot p$ and $H\cdot p$ coincide.

The same argument works for a point~$q$ over the vertex~$y$.   Thus we deduce
that the orbits of~$H$ and of~$H'$ through $p$ and $q$ coincide.  Therefore, $H
\cdot p$ and $H \cdot q$  are the only singular orbits  of the
cohomogeneity-one-action of~$H'$ on~$M$.

Take  any regular point~$o$  in the section~$\Sigma $  from above.  Then this
point $o$ must lie on a shortest $H'$-horizontal geodesic from $H\cdot  p$ to
$H \cdot q$. This geodesic is also $H$-horizontal, hence contained in~$\Sigma$.
For a generic choice of~$o$ in~$\Sigma$, we obtain a contradiction.
\end{proof}

\begin{remark}
If, under the assumptions of Proposition~\ref{PropCohOne}, the space~$M$ is
reducible and $H'$ acts with cohomogeneity one, then (possibly after enlarging
the group~$H$ without changing the orbit equivalence class) we are in the
following situation. The space $M$ splits as $M=M_1 \times M_2$ and $H$ splits
as $H_1\times H_2$, such that $H_i$ acts trivially on $M_{2-i}$. Moreover,
$H_2$ acts transitively on $M_2$ and $M_1$ is of rank one, see
\cite[Section~7]{lytchak}.
\end{remark}


\section{Subactions of transitive actions}
\label{SecSubTrans}


\begin{lemma}\label{LmSubCohOne}
Let $L$ be a simple compact Lie group of rank greater than one. Then there is
no closed subgroup $H \subset L \times L$ such that the $H$-action on~$L$ is
polar of cohomogeneity two with a non-flat section and such that the action of
$\pi_1(H) \times \pi_2(H)$ on~$L$ is transitive.
\end{lemma}
\begin{table}
\begin{tabular}{*{3}{|c}|}
\hline
$H_1$ & $L$ & $H_2$ \\
\hline\hline
$\Sp(n)$ & $\SU(2n)$ & $\SUxU{2n{-}1}{1}$ \\

 & & $\SU(2n{-}1)$ \\
\hline
$\SO(2n{-}1)$ & $\SO(2n)$ &  $\U(n)$ \\

 & & $\SU(n)$ \\
\hline
$\SO(4n{-}1)$ & $\SO(4n)$ &  $\Sp(n)\product\Sp(1)$ \\

 & & $\Sp(n)\product\U(1)$ \\

 & & $\Sp(n)$ \\
\hline
$\LG_2$ & $\SO(7)$ &  $\SO(6)$ \\
\hline
$\LG_2$ & $\SO(7)$ &  $\SO(5)\,\SO(2)$ \\

 & & $\SO(5)$ \\

\hline
$\Spin(7)$ & $\SO(8)$ &  $\SO(7)$ \\
\hline
$\Spin(9)$ & $\SO(16)$ &  $\SO(15)$ \\
\hline
\end{tabular}
\bigskip
\caption{Transitive actions.} \label{TTrans}
\end{table}

\begin{proof}
Consider the ideals $\h_1' := \ker (\pi_2|\h)$ and $\h_2' := \ker (\pi_1|\h)$
of~$\h$, where $\pi_1$ and $\pi_2$ are the natural projections onto the two
simple factors of the Lie algebra of~$L \times L$. Let $\h_\Delta$ be an ideal
which is complementary to $\h_1' + \h_2'$ in $\h$. Then we have $\pi_1(\h) =
\pi_1(\h_\Delta) \oplus \h_1'$ and $\pi_2(\h) = \pi_2(\h_\Delta) \oplus \h_2'$.
Since $\h$ is a proper subalgebra of $\pi_1(\h) \times \pi_2(\h)$, we have
$\h_\Delta \neq 0$.

All transitive actions of groups $H_1\times H_2$ on simple compact Lie
groups~$L$ given by~(\ref{EqLRAction}), where $H_1, H_2 \subset L$ are closed
proper subgroups, were determined by Oni\v{s}\v{c}ik~\cite{oniscik}; the result
is given in Table~\ref{TTrans}. Inspection of the table shows that for $L$
simple there is only one case where the Lie algebras of $H_1$ and $H_2$ have a
non-trivial isomorphic ideal, namely the case of the $\Spin(7) \times
\SO(7)$-action on~$\SO(8)$. However, in this case $\h = \h_\Delta$ would be
isomorphic to the $21$-dimensional Lie algebra $\so(7)$, obviously
contradicting the assumption that $H$ acts on the $28$-dimensional Lie
group~$\SO(8)$ with cohomogeneity two.
\end{proof}

\begin{remark}\label{RemTrans}
It should be noted that Table~\ref{TTrans} has to be interpreted on the Lie
algebra level. Indeed, e.g.\ the transitive actions on~$\SO(8)$ of $\Spin(7)
\times [\SO(6)\,\SO(2)]$ and of $\Spin(7) \times [\SO(5)\,\SO(3)]$ correspond
to the actions of $\SO(7) \times \U(4)$ and $\SO(7) \times [\Sp(2) \product
\Sp(1)]$ under a triality automorphism of~$\Spin(8)$, cf.\
\cite[Proposition~3.3]{hyperpolar}. In particular, some entries of
\cite[Table~7]{oniscik} do not appear in our Table~\ref{TTrans} exactly for
this reason.
\end{remark}


\section{Concluding remarks}
\label{SecConclusion}


It follows from the main result of~\cite{lytchak} that polar actions on compact
irreducible Riemannian symmetric spaces of higher rank are hyperpolar if the
cohomogeneity is greater than two. In this article, we have shown that polar
actions of cohomogeneity two on the classical compact Lie groups of higher rank
with biinvariant Riemannian metrics are hyperpolar. Without restriction on the
cohomogeneity, this was shown in~\cite{polar} for symmetric spaces of Type~I
and for exceptional symmetric spaces of Type~II in~\cite{exc}. Thus we have now
completed the proof of Theorem~\ref{ThMain}. In particular, this finally
settles the conjecture of Biliotti~\cite{biliotti} affirmatively.

It now follows from the results of~\cite{hyperpolar} that any polar action of a
compact Lie group on a compact irreducible  Riemannian symmetric space of
higher rank is of cohomogeneity one or orbit equivalent to a Hermann action (it
can be both, as there are many Hermann actions of cohomogeneity one).
See~\cite[Theorem~B]{hyperpolar} for the classification of cohomogeneity one
actions up to orbit equivalence. Let us recall the two types of Hermann actions
in the case where the symmetric space $G/K$ is isometric to a simple compact
Lie group~$L$ with biinvariant Riemannian metric. In this case we may assume $G
= L \times L$ and any involutive automorphism of~$G$ is either of the form
$(\ell_1,\ell_2) \mapsto (\sigma^{-1}(\ell_2), \sigma(\ell_1))$, where $\sigma$
is an automorphism of~$L$ (of arbitrary order) or of the form
$(\ell_1,\ell_2) \mapsto (\sigma(\ell_1), \tau(\ell_2))$ where both $\sigma$
and $\tau$ are involutive automorphisms of~$L$. In the first case, the fixed
point set of the involution is $\Delta^\sigma = \{ (\ell, \sigma(\ell)) \mid
\ell \in L \} \subset G$. In the latter case, the fixed point set is given by
the cartesian product of the fixed point sets $L^\sigma \times L^\tau$.

As stated in Lemma~\ref{LmNoDiag}, any subgroup of~$L \times L$ whose action
on~$L$ is orbit equivalent to the action of~$\Delta^\sigma$ is actually
conjugate to~$\Delta^\sigma$. However, in case of the action of $L^\sigma
\times L^\tau$ on~$L$, there are many examples of orbit equivalent subgroups,
see~\cite{polar}.

On the compact rank-one symmetric spaces, polar actions have been classified
in~\cite{pth1}.


\end{document}